\documentclass[12pt]{amsart}
\usepackage[unicode,bookmarks,colorlinks]{hyperref}

\usepackage{amsmath,amssymb,amsthm}
\usepackage{amsfonts}
\usepackage[mathscr]{eucal}
\usepackage{graphicx}
\usepackage{color}
\usepackage{amsmath,amssymb,amsthm}
\usepackage{amsfonts}
\usepackage[mathscr]{eucal}
\newtheorem{theorem}{Theorem}[section]
\newtheorem{lemma}[theorem]{Lemma}
\newtheorem{prop}[theorem]{Proposition}

\newtheorem{corollary}[theorem]{Corollary}

\theoremstyle{definition}
\newtheorem{defi}[theorem]{Definition}

\theoremstyle{remark}
\newtheorem{remark}[theorem]{Remark}

\numberwithin{equation}{section}

\newcommand{\rd}{{\mathbb R^d}}

\def\E{{\mathbb E}}

\begin{document}

\title{\bf Two-term trace estimates for relativistic stable processes}

\author{Rodrigo Ba\~{n}uelos}
\address{Rodrigo Ba\~{n}uelos, Department of Mathematics
Purdue University
150 North University Street
West Lafayete, Indiana 47907-2067
}
\email{banuelos@math.purdue.edu}
\urladdr{http://www.math.purdue.edu/~banuelos}

\author{Jebessa Mijena}
\address{Jebessa Mijena, 221 Parker Hall, Department of Mathematics and Statistics,
Auburn University, Auburn, Al 36849}
\email{jbm0018@tigermail.auburn.edu}

\author{Erkan Nane}
\address{Erkan Nane, 221 Parker Hall, Department of Mathematics and Statistics,
Auburn University, Auburn, Al 36849}
\email{nane@auburn.edu}
\urladdr{http://www.auburn.edu/$\sim$ezn0001}

\begin{abstract}

We prove trace estimates  for the relativistic $\alpha$-stable process extending the result of Ba\~{n}uelos and Kulczycki (2008) in the stable case.
\end{abstract}


\maketitle


\section{Introduction  and statement of main results}

For $m>0,$ an $\rd$-valued process with independent, stationary increments having the following characteristic function
$$
\E e^{i\xi\cdot X_t^{\alpha,m}}=e^{-t\{(m^{2/\alpha}+|\xi|^2)^{\alpha/2}-m\}},\ \ \xi\in \rd,
$$
is called relativistic $\alpha$-stable process with mass $m$. We assume that sample paths of $X_t^{\alpha,m}$ are right continuous and have left-hand limits a.s. If we put $m = 0$ we obtain the symmetric rotation invariant $\alpha-$stable process with the characteristic function $e^{-t|\xi|^\alpha}, \xi \in \mathbb R^d.$ We refer to this process as {isotropic }  $\alpha-$stable L\'evy process. For the rest of the paper we keep $\alpha, m$ and $d\geq 2$ fixed and drop $\alpha, m$ in the notation, when it does not lead to confusion. Hence from now on the relativistic $\alpha$-stable process is denoted by $X_t$ and its counterpart  {isotropic } $\alpha-$ stable  L\'evy process by $\tilde {X_t}$ . We keep this notational convention consistently throughout the paper, e.g., if $p_t(x-y)$ is the transition density of $X_t,$ then $\tilde p_t(x-y)$ is the transition density of $\tilde X_t$.

{In Ryznar \cite{ryznar} Green function estimates of the Sch\"odinger operator with the free Hamiltonian of the form
$$
(-\Delta +m^{2/\alpha})^{\alpha/2}-m,
$$
were investigated, where $m>0$ and $\Delta$ is the Laplace operator acting on $L^2(\rd)$.  Using the estimates in Lemma \ref{ryznar-bound} below and proof in Ba\~{n}uelos and Kulczycki (2008) we provide an extension of  the asymptotics in \cite{ban-kul} to the relativistic $\alpha$ stable processes for any $0<\alpha<2$.}

Brownian motion has characteristic function
$$
\E^0e^{i\xi \cdot B_t}=e^{-t|\xi|^2}, \ \ \xi\in \rd.
$$
Let $\beta =\alpha/2$.
Ryznar showed that $ X_t$  can be represented as a time-changed  Brownian motion.
Let $T_\beta(t), \ t>0$, denote the strictly $\beta$-stable subordinator with the following Laplace transform
\begin{equation}
\E^0e^{-\lambda  T_\beta(t)}=e^{-t\lambda^\beta}, \ \ \lambda>0.
\end{equation}
Let $\theta _\beta(t,u), \ u>0$, denote the density function of $T_\beta(t)$. Then the process $B_{T_\beta(t)}$ is the standard symmetric $\alpha$-stable process.

Ryznar \cite[Lemma 1]{ryznar} showed that we can obtain $X_t=B_{T_{\beta}(t,m)}$, where a subordinator $T_{\beta}(t,m)$ is a positive infinitely divisible process with stationary increments with probability density function
$$
\theta_\beta(t,u,m)=e^{-m^{1/\beta}u+mt}\theta _\beta(t,u), \ \ u>0.
$$

Transition density of $T_{\beta}(t,m)$ is given by $\theta_\beta(t,u-v,m)$. Hence the transition density of $X_t$ is $p(t, x, y) = p(t, x-y)$ given by
\begin{equation}\label{free-density}
p(t,x)=e^{mt}\int_0^\infty \frac{1}{(4\pi u)^{d/2}}e^{\frac{-|x|^2}{4u}} e^{-m^{1/\beta}u}\theta _\beta(t,u)du.
\end{equation}
{Then $$p(t, x, x) = p(t, 0) = e^{mt}\int_0^\infty \frac{1}{(4\pi u)^{d/2}} e^{-m^{1/\beta}u}\theta _\beta(t,u)du.$$} The function $p(t, x)$ is a radially symmetric decreasing and that
\begin{equation}\label{density-est}
p(t,x)\leq p(t,0)\leq e^{mt}\int_0^\infty \frac{1}{(4\pi u)^{d/2}}\theta _\beta(t,u)du=e^{mt}t^{-d/\alpha}\frac{\omega_d\Gamma(d/\alpha)}{(2\pi)^d\alpha},
\end{equation}
where $\omega_d=\frac{2\pi^{d/2}}{\Gamma(d/2)}$ is the surface area of the unit sphere in $\rd$.
For an open set $D$ in $\rd$ we define the first exit time from $D$ by $\tau_D=\inf \{ t\geq 0: \ \ X_t\notin D\}$.


We set
\begin{equation}\label{rd-definition}
r_{D}(t,x,y)=\E^x[p(t-\tau_D, X_{\tau_D}, y); \tau_D<t],
\end{equation}
and
\begin{equation}\label{rd-pd relation} p_D(t,x,y)=p(t,x,y)-r_{D}(t,x,y),\end{equation}
for any $x,y\in \rd$, $t>0$. For a nonnegative Borel function $f$ and $t>0$, let
$$
P_t^Df(x)=\E^x[f(X_t):\ t<\tau_D]=\int_D p_D(t,x,y)f(y)dy,
$$
be the semigroup of the killed process acting on $L^2(D)$, see, Ryznar \cite[Theorem 1]{ryznar}.

Let $D$ be a bounded domain (or of finite volume). Then the operator $P_t^D$ maps $L^2(D)$ into $L^\infty (D)$ for every $t>0$. This follows from \eqref{density-est}, \eqref{rd-definition}, and the general theory of heat semigroups as described in \cite{davies}. It follows that there exists an orthonormal basis of eigenfunctions $\{\varphi_n:\ n=1,\ 2,\ 3,\cdots\}$ for $L^2(D)$ and corresponding eigenvalues $\{\lambda_n:\ n=1,\ 2,\ 3,\cdots\}$ of the generator of the semigroup $P_t^D$ satisfying
$$\lambda_1<\lambda_2\leq \lambda_3\leq \cdots,$$
with $\lambda_n\to\infty$ as $n\to\infty$. By definition, the pair $\{\varphi_n, \lambda_n\}$ satisfies
$$
P_t^D\varphi_n(x)=e^{-\lambda_nt}\varphi_n(x), \ \ x\in D, \ t>0.$$
Under such assumptions we have
\begin{equation}\label{eigen-expansion}
p_D(t,x,y)=\sum_{n=1}^\infty e^{-\lambda_nt}\varphi_n(x)\varphi_n(y).
\end{equation}

In this paper we are interested in the behavior of  the trace of this semigroup
\begin{equation}\label{t-trace}
Z_D(t)=\int_D p_D(t,x,x)dx.
\end{equation}

Because of \eqref{eigen-expansion}  we can write \eqref{t-trace} as
\begin{equation}\label{partition-trace}
Z_D(t)=\sum_{n=1}^\infty e^{-\lambda_nt}\int_D\varphi^2_n(x)dx=\sum_{n=1}^\infty e^{-\lambda_nt}.
\end{equation}


We denote $d$-dimensional volume of $D$ by $|D|$.

Our first result is the Weyl's asymtotic for  the eigenvalues of the relativistic Laplacian
\begin{prop}\label{lim-zero-trace}
\begin{equation}\label{limit-trace}
\lim_{t\to 0}t^{d/\alpha}e^{-mt}Z_D(t)=C_1|D|,
\end{equation}
where $C_1=\frac{\omega_d\Gamma(d/\alpha)}{(2\pi)^d\alpha}.$
\end{prop}

Let $N(\lambda)$ be the number of eigenvalues $\{\lambda_j\}$ which do not exceed $\lambda$. It follows from \eqref{limit-trace} and the classical Tauberian theorem (see for example \cite{feller}, p.445 Theorem 2) where $L(t) = C_1 |D|e^{m/t}$ is our slowly varying function at infinity
 that
\begin{equation}\label{weyl1}
\lim_{\lambda\to \infty}\lambda ^{-d/\alpha}e^{-m/\lambda}N(\lambda )=\frac{C_1 |D|}{\Gamma(1+d/\alpha)}.
\end{equation}

This is the analogue for relativistic stable process of the celebrated Weyl's asymptotic formula for the eigenvalues of the Laplacian.

\begin{remark}  The first author presented \eqref{weyl1} at a conference in Vienna at the Scrh\"odingier Institute in 2009 (see \cite{ban1}) and at the 34th conference in stochastic processes and their applications in Osaka in 2010 (see \cite{ban2}). Thus this result has been known to the authors, and perhaps to others, for number of years.
\end{remark}

 Our goal in this paper is to obtain the second term in the asymptotics of $Z_D (t)$ under some additional assumptions on the smoothness of $D$. Our result is inspired by result for trace estimates for stable processes by Ba\~nuelos and  Kulczycki \cite{ban-kul}.

 To state our main result we need the following property of the domain $D$.

 \begin{defi}
  The boundary, $\partial D$, of an open set $D$ in $\mathbb{R}^d$ is said to be $R-$smooth if for each point $x_0 \in \partial D$ there are two open balls $B_1$ and $B_2$ with radii $R$ such that $B_1 \subset D, B_2 \subset \mathbb{R}^d\backslash (D \cup \partial D)$ and $ \partial B_1 \cap \partial B_2 = x_0.$
 \end{defi}

 \begin{theorem}\label{main-theorem}
 Let $D \subset \mathbb{R}^d, d\geq 2,$ be an open bounded set with $R-smooth$ boundary. Let $|D|$ denote the volume ($d-$dimensional Lebesgue measure) of $D$ and $|\partial D|$ denote its surface area ($(d-1)-$dimensional  Lebesgue measure) of its boundary. Suppose $\alpha \in (0, 2).$ Then
 \begin{equation}\label{main-statement} \left|Z_{D}(t)-\frac{C_1 (t)e^{mt}|D|}{t^{d/\alpha}} + C_2(t) |\partial D|\right|\leq \frac{C_3e^{2mt}|D|t^{2/\alpha}}{R^{2} t^{d/\alpha}}, \ t>0,
\end{equation}
  where
 $$C_{1}(t) = \frac{1}{(4\pi)^{d/2}}\int_{0}^{\infty}z^{-d/2}e^{-(mt)^{1/\beta}z}\theta_{\beta}(1,z)dz\rightarrow C_1 = \frac{\omega_d\Gamma(d/\alpha)}{(2\pi)^d\alpha},  \ \ \ \mathrm{as}\  t \rightarrow 0,$$
 $$C_{2}(t) = \int_{0}^{\infty} r_{H}(t, (x_{1}, 0,\cdots, 0), (x_{1}, 0,\cdot \cdot \cdot, 0))dx_1\leq \frac{C_4 e^{2mt}t^{1/\alpha}}{t^{d/\alpha}} ,\hspace{.4cm}t > 0$$
 $$
 C_4=\int_{0}^{\infty} \tilde{r}_{H}(1, (x_{1}, 0,\cdots, 0), (x_{1}, 0,\cdot \cdot \cdot, 0))dx_1,
 $$
 $C_3 = C_{3}(d, \alpha), H = \{(x_1, \cdot \cdot \cdot, x_d)\in {\mathbb{R}^d}: x_1 > 0\}$ and ${r_H}$ is given by \eqref{rd-definition}.

\end{theorem}

 \begin{remark}When $m=0$, $0< \alpha \leq 2$, $C_2(t)=C_4 t^{1/\alpha}/t^{d/\alpha}$. Then  the result  in Theorem \ref{main-theorem} becomes for bounded domains with $R-$smooth boundary
 \begin{equation}
 \left |Z_D(t)-\frac{C_1 |D|}{t^{d/\alpha}}+\frac{C_4 |\partial D|t^{1/\alpha}}{t^{d/\alpha}}\right|\leq \frac{C_7 | D|t^{2/\alpha}}{R^2t^{d/\alpha}},
 \end{equation}
 where  $C_1, C_4$  are as in Theorem \ref{main-theorem}.
 This was established by   Ba\~{n}uelos and  Kulczycki \cite {ban-kul} recently.
\end{remark}

 The asymptotic for the trace of the heat kernel when $\alpha = 2$ (the case of the Laplacian with Dirichlet boundary condition in a domain of $\mathbb{R}^d$), have been extensively studied by many authors. For Brownian motion Van den Berg \cite{van}, proved that under the $R-$ smoothness condition
\begin{equation} \left|Z_{D}(t)-(4\pi t)^{-d/2}\left(|D|-\frac{\sqrt{\pi t}}{2}|\partial D|\right)\right|\leq \frac{C_d|D|t^{1-d/2}}{R^2},\ \ \ t>0.
\end{equation}
 For domains with $C^1$ boundaries the result
 \begin{equation}Z_D(t) = (4\pi t)^{-d/2}\left(|D|-\frac{\sqrt{\pi t}}{2}|\partial D|+o(t^{1/2})\right),\ \ \ \mathrm{as}\  t\rightarrow 0,
 \end{equation} was proved by Brossard and Carmona \cite {bro-carm}, for Brownian motion.

 \section{Preliminaries}
 { Let the ball in $\mathbb R^d$ with center at $x$ and radius $r,\{y:|y-x|<r\},$ be denoted by $B(x, r).$ We will use $\delta_D(x)$ to denote the Euclidean distance between $x$ and the boundary, $\partial D$, of $D$. That is, $\delta_D(x)$ = dist$(x, \partial D)$}.  Define
 $$\psi (\theta) = \int_{0}^{\infty} e^{-v}v^{p-1/2}(\theta + v/2)^{p-1/2}dv, \theta \geq 0,$$ where $p = (d+\alpha)/2.$ We put $\mathcal R(\alpha, d) = \mathcal A(-\alpha, d)/\psi(0),$ where $\mathcal A(v, d) = (\Gamma ((d-v)/2))/(\pi^{d/2}2^v|\Gamma(v/2)|). $ Let $\nu (x), \tilde \nu (x)$
 be the densities of the L\'evy measures of the relativistic $\alpha-$stable process and the standard $\alpha-$stable process, respectively. These densities are given by
 \begin{equation}
 \nu(x) = \frac{\mathcal R(\alpha, d)}{|x|^{d+\alpha}}e^{-m^{1/\alpha}|x|}\psi(m^{1/\alpha}|x|),
 \end{equation}
 and
 \begin{equation}
 \tilde{v}(x) = \frac{\mathcal A(-\alpha, d)}{|x|^{d+\alpha}}.
 \end{equation}

 We need the following estimate of the transition probabilities of the process $X_t$ which is given in (\cite{siudeja-kul}, Lemma 2.2): For any $x, y\in \mathbb R^d$ and $t>0$ there exist constants $c_1>0$ and $c_2>0,$
\begin{equation}\label{p-estimate}
 p(t, x, y) \leq c_1e^{mt} \min \bigg\{ \frac{t}{|x-y|^{d+\alpha}}e^{-c_2|x-y|}, t^{-d/\alpha} \bigg\}.
\end{equation}
 We will also use the fact(\cite{bog dan}, Lemma 6) that if $D \subset \mathbb R^d$ is an open bounded set satisfying a uniform outer cone condition, then $P^x(X(\tau_D)\in \partial D) = 0$ for all $x\in D.$ For the open bounded set $D$ we will denoted by $G_D(x, y)$ the Green function for the set D equal to, $$G_D(x, y) = \int_0^\infty p_D(t, x, y)dt,\ \ \ x,y \in \mathbb R^d.$$
For any such $D$ the expectation of the exit time of the processes $X_t$ from $D$ is given by the integral of the Green function over the domain. That is: $$E^x(\tau_D) = \int_D G_D(x,y)dy.$$
\begin{lemma}\label{rd-estimate}
Let $D\subset \mathbb R^d$ be an open set. For any $x, y\in D$ we have

$$r_D(t, x, y)\leq c_1e^{mt}\left(\frac{t}{\delta_D^{d+\alpha}(x)}e^{-c_2\delta_D(x)} \wedge t^{-d/\alpha}\right).$$

\end{lemma}
\begin{proof}
Using \eqref{rd-definition} and \eqref{p-estimate} we have
\begin{eqnarray}
r_D(t, x, y) &=& E^y(p(t-\tau_D, X(\tau_D),x); \tau_D < t) \nonumber\\
&\leq& c_1e^{mt}E^y\left(\frac{t}{|x-X(\tau_D)|^{d+\alpha}}e^{-c_2|x-X(\tau_D)|}\wedge t^{-d/\alpha}\right)\nonumber\\&\leq& c_1e^{mt}\left(\frac{t}{\delta_D^{d+\alpha}(x)}e^{-c_2\delta_D(x)}\wedge t^{-d/\alpha}\right)\nonumber.
\end{eqnarray}
\end{proof}

We need the following result for the proof of Proposition \ref{lim-zero-trace}.
\begin{lemma}
\begin{equation}\label{free-lim-zero}
\lim_{t\to 0} p(t,0)e^{-mt}t^{d/\alpha}=C_1,
\end{equation}
where
$$
C_1=(4\pi)^{d/2}\int_{0}^\infty u^{-d/2}\theta_\beta(1,u)du=\frac{\omega_d\Gamma(d/\alpha)}{(2\pi)^d\alpha}.
$$
\end{lemma}
\begin{proof}
By \eqref{free-density} we have
$$
p(t,x, x)=p(t,0)=e^{mt}\int_0^\infty \frac{1}{(4\pi u)^{d/2}} e^{-m^{1/\beta}u}\theta _\beta(t,u)du.
$$

Now using the scaling of stable subordinator  $\theta _\beta(t,u)=  t^{-1/\beta}\theta _\beta(1,ut^{-1/\beta})$ and a change of variables we get
$$
p(t,0)=\frac{e^{mt}}{(4\pi)^{d/2}t^{d/\alpha}}\int_0^\infty z^{-d/2}e^{-m^{1/\beta}t^{1/\beta}z}\theta _\beta(1,z)dz = \frac{C_1(t)e^{mt}}{t^{d/\alpha}},
$$
then by dominated convergence theorem we obtain
$$
\lim_{t\to 0} p(t,0)e^{-mt}t^{d/\alpha}=\frac{1}{(4\pi)^{d/2}}\int_0^\infty z^{-d/2}\theta _\beta(1,z)dz,
$$
and this last integral is equal to the density of  $\alpha$-stable process at time $1$ and $x=0$ which was calculated in \cite{ban-kul} to be
$$
\frac{\omega_d\Gamma(d/\alpha)}{(2\pi)^d\alpha}.
$$
\end{proof}

We next give the proof of Proposition \ref{lim-zero-trace}.

\begin{proof}[\bf Proof of Proposition \ref{lim-zero-trace}]
By  \eqref{rd-definition} we see that
\begin{equation}
\frac{p_D(t,x,x)}{C_1e^{mt}t^{-d/\alpha}}=\frac{p(t,0)}{C_1e^{mt}t^{-d/\alpha}}-
\frac{r_D(t,x,x)}{C_1e^{mt}t^{-d/\alpha}}.
\end{equation}
Since the first term tend to $1$ as $t\to 0$ by \eqref{free-lim-zero}, in order to prove \eqref{limit-trace}, we show that
\begin{equation}
\frac{t^{d/\alpha}}{C_1e^{mt}}\int_D r_D(t,x,x)dx
\to 0, \ \ \mathrm{as} \ t\to 0.
\end{equation}
{ For $q\geq 0$, we define $D_q = \{x\in D: \delta_{D}(x) \geq q\}$. Then for $0 < t < 1,$ consider the subdomain $D_{t^{1/2\alpha}}=\{x\in D: \  \delta_D(x)\geq t^{1/2\alpha}\}$ and its complement $D_{t^{1/2\alpha}}^C=\{x\in D: \  \delta_D(x)< t^{1/2\alpha}\}$.  Recalling that $|D|<\infty$, by Lebesgue dominated convergence theorem we get $|D_{t^{1/2\alpha}}^C|\to 0,$ as $t\to 0$.} Since $p_D(t,x,x)\leq p(t,x,x)$, by \eqref{density-est} we see that
 $$
 \frac{r_D(t,x,x)}{C_1e^{mt}t^{-d/\alpha}}\leq 1,
 $$
 for all $x\in D$. It follows that
 \begin{equation}
\frac{t^{d/\alpha}}{C_1e^{mt}}\int_{D^{C}_{t^{1/2\alpha}}} r_D(t,x,x)dx
\to 0, \ \ \mathrm{as} \ t\to 0.
\end{equation}
 On the other hand, by Lemma 2.2 in \cite{siudeja-kul} we obtain
 \begin{eqnarray}
 \frac{r_D(t,x,x)}{C_1e^{mt}t^{-d/\alpha}}&=& \frac{\E^x[p(t-\tau_D, X_{\tau_D}, x);t\geq \tau_D]}{C_1e^{mt}t^{-d/\alpha}}\nonumber\\
 &\leq& c\E^y\min \bigg\{ \frac{t^{1+d/\alpha}}{|x-X(\tau_D)|^{d+\alpha}}e^{-c_2|x-X(\tau_D)|}, 1 \bigg\}\nonumber\\
 &\leq& c\min \bigg\{ \frac{t^{1+d/\alpha}}{\delta_D(x)^{d+\alpha}}e^{-c_2\delta_D(x)}, 1 \bigg\}.\label{upper-complement}
 \end{eqnarray}
 For $x\in D_{t^{1/2\alpha}}$ and $0<t<1$, the right hand side of \eqref{upper-complement} is bounded above by $ct^{d/2\alpha+1/2}$ and hence

 \begin{equation}
 \frac{t^{d/\alpha}}{C_1e^{mt}}\int_{D_t^{1/2\alpha}} r_D(t,x,x)dx\leq ct^{d/2\alpha+1/2}|D|,
 \end{equation}
 and this last quantity goes to $0$ as $t\to 0$.
\end{proof}

 For an open set $D\subset \mathbb R^d$ and $x\in \mathbb R^d,$ the distribution $P^x(\tau_D < \infty, X(\tau_D)\in \cdot)$ will be
 called the relativistic $\alpha-$harmonic measure for $D.$ The following Ikeda-Watanabe formula recovers the relativistic $\alpha-$harmonic measure for the set $D$ from the Green function.
\begin{prop} [\cite{siudeja-kul}]
Assume that $D$ is an open, nonempty, bounded subset of $\mathbb R^d,$ and $A$ is a Borel set such that dist$(D, A)>0.$ Then
\begin{equation}\label{IW-green}
P^x(X(\tau_D)\in A, \tau_D < \infty) = \int_DG_D(x, y)\int_Av(y-z)dzdy,\ \ \  x\in D.
\end{equation}
\end{prop}
Here we need the following generalization already stated and used in \cite{ban-kul}.
\begin{prop}\label{ik-wa generalization}\cite[Proposition 2.5]{siudeja-kul}
Assume that $D$ is an open, nonempty, bounded subset of $R^d,$ and $A$ is a Borel set such that $A \subset D^c \backslash \partial D$ and $0 \leq t_1 < t_2 < \infty, x\in D.$ Then we have
$$P^x(X(\tau_D)\in A, t_1 < \tau_D < t_2) = \int_D\int_{t_1}^{t_2}p_D(s, x, y)ds\int_Av(y-z)dzdy.$$
\end{prop}
The following propostition holds for a large class of L\'evy processes
\begin{prop}\label{pf-pd}\cite[Proposition 2.3]{ban-kul}
Let $D$ and $F$ be open sets in $\mathbb R^d$ such that $D\subset F.$ Then for any $x, y\in \mathbb R^d$ we have
$$p_F(t, x, y) - p_D(t, x, y) = E^x(\tau_D < t, X(\tau_D)\in F \backslash D; p_F(t-\tau_D, X(\tau_D), y)).$$
\end{prop}

\begin{lemma}\label{ryznar-bound}\cite[Lemma 5]{ryznar}
Let $D\subset \rd$ be an open set. For any  $x,y\in D$ and $t>0$ the following estimates hold;
\begin{equation}
\begin{split}
p_D(t,x,y)&\leq e^{mt}\tilde{p}_D(t,x,y)\\
r_D(t,x,y)&\leq e^{2mt}\tilde{r}_D(t,x,y).
\end{split}
\end{equation}
\end{lemma}
%

 We need the following lemma given by Van den Berg in \cite{van}.
\begin{lemma}\cite[Lemma 5]{van} Let $D$ be an open bounded set in $R^d$ with R-smooth boundary $\partial D$ and for $0\leq q < R$
denote the area of boundary of $\partial D_q$ by $|\partial D_q|$. Then
\begin{equation}
\bigg( \frac{R-q}{R}\bigg)^{d-1}|\partial D|\leq |\partial D_q|\leq \bigg(\frac{R}{R-q}\bigg)^{d-1}|\partial D|,\ \ 0\leq q < R.
\end{equation}
\end{lemma}
\begin{corollary}\label{d-boundary}
(\cite{ban-kul}, Corollary 2.14) Let $D$ be an open bounded set in $\mathbb R^d$ with R-smooth boundary. For any $0 < q \leq R$ we have

(i)
$$2^{-d+1}|\partial D| \leq |\partial D_q| \leq 2^{d-1}|\partial D|,$$
(ii)
$$|\partial D| \leq \frac{2^d|D|}{R},$$
(iii)
$$\bigg||\partial D_q| - |\partial D|\bigg|\leq \frac{2^ddq|\partial D|}{R}\leq \frac{2^{2d}dq|D|}{R^2}.$$
\end{corollary}
\section{Proof of main result}
\begin{proof}[\bf  Proof of Theorem \ref{main-theorem}]
For the case $t^{1/\alpha} > R/2$ the theorem holds trivially. Indeed, by Equation \eqref{density-est}
$$Z_D(t)\leq \int_Dp(t, x, x)dx \leq \frac{c_1e^{mt}|D|}{t^{d/\alpha}}\leq \frac{c_1e^{mt}|D|t^{2/\alpha}}{R^2t^{d/\alpha}}.$$
By Corollary \ref{d-boundary}  and Lemma \ref{ryznar-bound}  we also have
$$C_2(t)|\partial D|\leq \frac{C_4e^{2mt}|\partial D|t^{1/\alpha}}{t^{d/\alpha}}\leq \frac{2^dC_4e^{2mt}|D|t^{1/\alpha}}{Rt^{d/\alpha}}\leq\frac{2^{d+1}C_4e^{2mt}|D|t^{2/\alpha}}{R^2t^{d/\alpha}}$$
$$
\frac{C_1(t)e^{mt}|D|}{t^{d/\alpha}}\leq \frac{C_1e^{mt}|D|t^{2/\alpha}}{R^2t^{d/\alpha}}.
$$
Therefore for $t^{1/\alpha} > R/2$ \eqref{main-statement} holds.
 Here and in sequel we consider the case $t^{1/\alpha}\leq R/2$. From \eqref{rd-pd relation} and the fact that $p(t, x, x) = \frac{C_1(t)e^{mt}}{t^{d/\alpha}},$ we have that
 \begin{eqnarray}\label{zpr-estimate}Z_D(t)-\frac{C_1(t)e^{mt}|D|}{t^{d/\alpha}} &=& \nonumber\int_Dp_D(t, x, x)dx - \int_Dp(t, x, x)dx\\&=&  -\int_Dr_D(t, x, x)dx \label{zd-rd equation},
 \end{eqnarray}
where $C_1(t)$ is as stated in the theorem. Therefore we must estimate \eqref{zd-rd equation}. We break our domain into two pieces, $D_{R/2}$ and its complement $D^C_{R/2}$. We will first consider the contribution of $D_{R/2}.$
\\
\\
\noindent {\bf Claim 1:} For $t^{1/\alpha} \leq R/2$ we have
 \begin{equation}\label{rd-estimate}
 \int_{D_{R/2}}r_D(t, x, x)dx \leq \frac{ce^{2mt}|D|t^{2/\alpha}}{R^2t^{d/\alpha}}.
 \end{equation}

\noindent{\bf Proof of Claim 1:}
 By Lemma \ref{ryznar-bound} we have
\begin{equation}\label{rz-estimate}\int_{D_{R/2}}r_D(t, x, x)dx\leq e^{2mt}\int_{D_{R/2}}\tilde r_D(t, x, x)dx,
\end{equation}
and by scaling of the stable density the right hand side of \eqref{rz-estimate} equals
\begin{equation}
\frac{e^{2mt}}{t^{d/\alpha}}\int_{D_{R/2}}\tilde r_{D/t^{1/\alpha}}(1, \frac{x}{t^{1/\alpha}}, \frac{x}{t^{1/\alpha}})dx.
\end{equation}
For $x\in D_{R/2}$ we have $ \delta_{D/t^{1/\alpha}}(x/t^{1/\alpha})\geq R/(2t^{1/\alpha})\geq 1.$ By \cite[Lemma 2.1]{ban-kul}, we get
$$\tilde r_{D/t^{1/\alpha}}\bigg(1, \frac{x}{t^{1/\alpha}}, \frac{x}{t^{1/\alpha}}\bigg)\leq \frac{c}{ \delta^{d+\alpha}_{D/t^{1/\alpha}}(x/t^{1/\alpha})}\leq \frac{c}{ \delta^{2}_{D/t^{1/\alpha}}(x/t^{1/\alpha})}\leq \frac{ct^{2/\alpha}}{R^2}.$$
 Using the above inequality, we get
 $$\int_{D_{R/2}}r_D(t, x, x)dx \leq \frac{e^{2mt}}{t^{d/\alpha}}\int_{D_{R/2}}\frac{ct^{2/\alpha}}{R^2}dx\leq \frac{ce^{2mt}|D|t^{2/\alpha}}{R^2t^{d/\alpha}},$$
which proves \eqref{rd-estimate}.

  Now we will introduce the following notation. Since $D$ has $R$-smooth boundary, for any point $y\in \partial D$ there are two open balls $B_1$ and $B_2$ both of radius $R$ such that $B_1 \subset D, B_2\subset \mathbb{R}^d\backslash (D\cup \partial D), \partial B_1 \cap \partial B_2 = y.$ For any $x\in D_{R/2}$ there exist a unique point $x_* \in \partial D$ such that $\delta_D(x) = |x-x_*|.$ Let $B_1 = B(z_1, R), B_2 = B(z_2, R)$ be inner/outer balls for the point $x_*.$ Let $H(x)$ be the half-space containing $B_1$ such that $\partial H(x)$ contains $x_*$ and is perpendicular to the segment $\overline {z_1z_2}.$

  We will need the following very important proposition in the proof of Theorem 1.2. Such a proposition has been proved for the stable process in \cite[Proposition 3.1]{ban-kul}.
\begin{prop}\label{main-prop}
Let $D\subset R^d, d\geq 2,$ be an open bounded set with R-smooth boundary $\partial D$. Then for any $x\in D^C_{R/2}$ and $t > 0$ such that $t^{1/\alpha} \leq R/2$ we have
\begin{equation}\label{rdh-estimate}
|r_D(t, x, x) - r_{H(x)}(t, x, x)|\leq \frac{ce^{2mt}t^{1/\alpha}}{Rt^{d/\alpha}}\bigg(\bigg(\frac{t^{1/\alpha}}{\delta_D(x)}\bigg)^{d+\alpha /2 -1}\wedge 1\bigg).
\end{equation}
\end{prop}
{\begin{proof}
 Exactly as in \cite{ban-kul}, let $x_{*}\in \partial D$ be a unique point such that $|x - x_{*}| = \text{dist}(x, \partial D)$ and $B_1$ and $B_2$ be balls with radius $R$ such that $B_1 \subset D, B_2 \subset \mathbb{R}^d \backslash (D\cup \partial D), \partial B_1\cap \partial B_2 = x_*.$ Let us also assume that $x_* = 0$ and choose an orthonormal coordinate system $(x_1,x_2,...,x_d)$ so that the positive axis $0x_1$ is in the direction of $\overrightarrow{0p}$ where $p$ is the center of the ball $B_1.$ Note that $x$ lies on the interval $0p$ so $x = (|x|, 0, 0, ..., 0).$ Note also that $B_1\subset D\subset (\overline{B_2})^{c}$ and $B_1\subset H(x)\subset (\overline{B_2})^{c}.$ For any open sets $A_1, A_2$ such that $A_1\subset A_2$ we have $r_{A_{1}}(t, x, y) \geq r_{A_{2}}(t, x, y)$ so
$$|r_D(t, x, x) - r_{H(x)}(t, x, x)|\leq r_{B_1}(t, x, x) - r_{(\overline{B_2})^{c}}(t, x, x).$$
So in order to prove the proposition it suffices to show that
$$r_{B_1}(t, x, x) - r_{(\overline{B_2})^{c}}(t, x, x)\leq \frac{ce^{2mt}t^{1/\alpha}}{Rt^{d/\alpha}}\bigg(\bigg(\frac{t^{1/\alpha}}{\delta_D(x)}\bigg)^{d+\alpha /2 -1}\wedge 1\bigg),$$
for any $x = (|x|, 0, ..., 0), |x|\in (0, R/2].$  Such an estimate was proved for the case $m=0$ in \cite{ban-kul}.
 In order to complete the proof it is enough to prove that

$$r_{B_1}(t, x, x) - r_{(\overline{B_2})^{c}}(t, x, x)\leq ce^{2mt}\left\{\tilde{r}_{B_1}(t, x, x) - \tilde{r}_{(\overline{B_2})^{c}}(t, x, x)\right\}.%
$$

To show this given the ball $B_2$, we set $U = (\overline{B_2})^{c}.$ Now using the generalized Ikeda-Watanabe formula, Proposition \eqref{pf-pd} and Lemma \eqref{ryznar-bound} 
 we have
\begin{eqnarray}
 &&r_{B_1}(t, x, x) - r_{U}(t, x, x)\nonumber\\
&=& E^{x}\left[t > \tau_{B_1}, X({\tau_{B_1}})\in U\backslash B_1; p_{U}(t-\tau_{B_1}, X({\tau_{B_1}}), x)\right]\nonumber\\
&=&\int_{B_1}\int_{0}^{t}p_{B_1}(s, x, y)ds\int_{U\backslash B_1}v(y - z)p_{U}(t-s, z, x)dzdy\nonumber\\
&\leq&e^{2mt}\int_{B_1}\int_{0}^{t}\tilde p_{B_1}(s, x, y)ds\int_{U\backslash B_1}\tilde v(y-z)\tilde p_{U}(t-s, z, x)dzdy\nonumber\\
&\leq& ce^{2mt}E^x\left[t> \tilde \tau_{B_1}, \tilde X(\tilde \tau_{B_1})\in U\backslash B_1; \tilde p_{U}(t - \tilde \tau_{B_1}, \tilde X(\tilde \tau_{B_1}), x)\right]\nonumber\\
&=&ce^{2mt}\left(\tilde r_{B_1}(t, x, x) - \tilde r_{U}(t, x, x)\right)\nonumber\\
&\leq& \frac{ce^{2mt}t^{1/\alpha}}{Rt^{d/\alpha}}\bigg(\bigg(\frac{t^{1/\alpha}}{\delta_D(x)}\bigg)^{d+\alpha /2 -1}\wedge 1\bigg)\nonumber.
\end{eqnarray}
The last inequality follows by Proposition 3.1 in \cite{ban-kul}.
\end{proof}}
  Now using this proposition we estimate the contribution from $D\backslash D_{R/2}$ to the integral of $r_D(t, x, x)$ in \eqref{zd-rd equation}.
\\

  \noindent{\bf Claim 2:} For $t^{1/\alpha}\leq R/2$ we get

\begin{equation}\label{rh-estimate}
\left|\int_{D\backslash D_{R/2}}r_D(t, x, x)dx-\int_{D\backslash D_{R/2}}r_{H(x)}(t, x, x)dx\right|\leq \frac{ce^{2mt}|D|t^{2/\alpha}}{R^2t^{d/\alpha}}.
\end{equation}

\noindent{\bf Proof of Claim 2:}
By Proposition \ref{main-prop} the left hand side of \eqref{rh-estimate} is bounded above by
$$\frac{ce^{2mt}t^{1/\alpha}}{Rt^{d/\alpha}}\int_0^{R/2}|\partial D_q|\bigg(\bigg(\frac{t^{1/\alpha}}{q}\bigg)^{d+\alpha/2-1}\wedge 1\bigg)dq.$$
By Corollary \ref{d-boundary}, (i), the last quantity is smaller than or equal to
$$\frac{ce^{2mt}t^{1/\alpha}|\partial D|}{Rt^{d/\alpha}}\int_0^{R/2}\bigg(\bigg(\frac{t^{1/\alpha}}{q}\bigg)^{d+\alpha/2-1}\wedge 1\bigg)dq.
$$
The integral in the last quantity is bounded by $ct^{1/\alpha}.$ To see this observe that since $t^{1/\alpha} \leq R/2$ the above integral is equal to
 \begin{eqnarray}&&\int_0^{t^{1/\alpha}}\bigg(\bigg(\frac{t^{1/\alpha}}{q}\bigg)^{d+\alpha/2-1}\wedge 1\bigg)dq + \int_{t^{1/\alpha}}^{R/2}\bigg(\bigg(\frac{t^{1/\alpha}}{q}\bigg)^{d+\alpha/2-1}\wedge 1\bigg)dq \nonumber \\
 &=& \int_0^{t^{1/\alpha}} 1dq + \int_{t^{1/\alpha}}^{R/2}\bigg(\frac{t^{1/\alpha}}{q}\bigg)^{d+\alpha/2-1}dq\nonumber \\
 &\leq&  ct^{1/\alpha}.\nonumber
\end{eqnarray}
Using this and Corollary \eqref{d-boundary}, (ii), we get \eqref{rh-estimate}.

 Recall that $H = \lbrace(x_1, \cdots, x_d)\in \mathbb{R}^d: x_1 > 0\rbrace$. For abbreviation let us denote
 $$f_H(t, q) = r_H(t, (q, 0, \cdots, 0), (q, 0, \cdots, 0)), \hspace{0.3cm} t, q > 0.$$
Of course we have $r_{H(x)}(t,x,x) = f_H(t, \delta_{H}(x)).$ In the next step we will show that
 \begin{equation}\label{hf-estimate}
\left|\int_{D\backslash D_{R/2}}r_{H(x)}(t, x, x)dx- |\partial D|\int_{0}^{R/2}f_{H}(t, q)dq\right|\leq \frac{ce^{2mt}|D|t^{2/\alpha}}{R^2t^{d/\alpha}}.
\end{equation}
We have
$$\int_{D\backslash D_{R/2}}r_{H(x)}(t, x, x)dx = \int_{0}^{R/2}|\partial D_q|f_{H}(t, q)dq.$$
Hence the left hand side of (\ref{hf-estimate}) is bounded above by
$$\int_{0}^{R/2}\left||\partial D_q| - |\partial D|\right|f_H(t,q)dq.$$
By Corollary \ref{d-boundary}, (iii), this is smaller than
\begin{eqnarray}&&\frac{c|D|}{R^2}\int_{0}^{R/2}qf_H(t,q)dq\nonumber\\
&\leq& \frac{c|D|e^{2mt}}{R^2}\int_{0}^{R/2}q\tilde f_H(t,q)dq\nonumber\\
&=& \frac{c|D|e^{2mt}}{R^2}\int_{0}^{R/2}qt^{-d/\alpha}\tilde f_H(1,qt^{-1/\alpha})dq\nonumber\\
&=& \frac{c|D|e^{2mt}}{R^2t^{d/\alpha}}\int_{0}^{R/2t^{1/\alpha}}qt^{2/\alpha}\tilde f_H(1,q)dq\nonumber\\
&\leq& \frac{c|D|e^{2mt}t^{2/\alpha}}{R^2t^{d/\alpha}}\int_{0}^{\infty}q\left(q^{-d-\alpha}\wedge 1\right)dq\nonumber \leq \frac{c|D|e^{2mt}t^{2/\alpha}}{R^2t^{d/\alpha}}.
\end{eqnarray}
This shows (\ref{hf-estimate}). Finally, we have
\begin{eqnarray}
&&\bigg||\partial D|\int_{0}^{R/2}f_H(t,q)dq - |\partial D|\int_{0}^{\infty}f_H(t,q)dq\bigg|\nonumber\\
&\leq& |\partial D|\int_{R/2}^{\infty}f_H(t,q)dq\nonumber\\
&\leq& \frac{c|D|}{R}\int_{R/2}^{\infty}f_H(t,q)dq\hspace{.4cm}\text{by Corollary \ref{d-boundary}, (ii)}\nonumber\\
&\leq&\frac{c|D|e^{2mt}}{Rt^{d/\alpha}}\int_{R/2}^{\infty}\tilde f_H(1,qt^{-1/\alpha})dq\nonumber\\
&=&\frac{c|D|e^{2mt}t^{1/\alpha}}{Rt^{d/\alpha}}\int_{R/2t^{1/\alpha}}^{\infty}\tilde f_H(1,q)dq\nonumber.
\end{eqnarray}
 Since $R/2t^{1/\alpha} \geq 1,$ so for $q \geq R/2t^{1/\alpha}\geq 1$ we have $\tilde f_H(1,q)\leq cq^{-d-\alpha}\leq cq^{-2}.$ Therefore,
 $$\int_{R/2t^{1/\alpha}}^{\infty}\tilde f_H(1,q)dq \leq c\int_{R/2t^{1/\alpha}}^{\infty}\frac{dq}{q^2} \leq \frac{ct^{1/\alpha}}{R}.$$
 Hence,
 \begin{equation}\label{f-estimate}\bigg||\partial D|\int_{0}^{R/2}f_H(t,q)dq - |\partial D|\int_{0}^{\infty}f_H(t,q)dq\bigg|\leq \frac{c|D|e^{2mt}t^{2/\alpha}}{R^2t^{d/\alpha}}.
 \end{equation}
 \\

 Note that the constant $C_2(t)$ which appears in the formulation of Theorem \ref{main-theorem} satisfies $C_2(t) = \int_{0}^{\infty}f_H(t,q)dq.$
  Now  equations \eqref{zpr-estimate}, \eqref{rd-estimate}, \eqref{rh-estimate}, \eqref{hf-estimate}, \eqref{f-estimate} give (\ref{main-statement}).
\end{proof}

\end{document}